\documentclass[submission,copyright,creativecommons]{eptcs}
\providecommand{\event}{GASCom 2026} 

\usepackage{iftex}

\ifpdf
  \usepackage{underscore}         
  \usepackage[T1]{fontenc}        
\else
  \usepackage{breakurl}           
\fi

\usepackage[utf8]{inputenc}
\usepackage[english]{babel}

\usepackage{amsfonts}
\usepackage{amsmath}
\usepackage{amssymb}
\usepackage{amsthm}
\usepackage{hyperref}
\usepackage[usenames]{color}

\usepackage{caption}
\usepackage{algorithm}
\usepackage{etoolbox}\AtBeginEnvironment{algorithmic}{\small} 
\usepackage[noend]{algpseudocode}
\usepackage{mathrsfs}

\usepackage{tikz}
\usetikzlibrary{
	calc,
	intersections,
	through,
	backgrounds,
	arrows
}

\theoremstyle{definition}
\newtheorem{theorem}{Theorem}[section]
\newtheorem{lemma}[theorem]{Lemma}
\newtheorem{proposition}[theorem]{Proposition}
\newtheorem{corollary}[theorem]{Corollary}
\newtheorem{definition}[theorem]{Definition}

\newtheorem{fact}[theorem]{Fact}

\theoremstyle{remark}
\newtheorem{remark}[theorem]{Remark}

\renewcommand{\le}{\leqslant}
\renewcommand{\ge}{\geqslant}

\newcommand{\La}{\Lambda}

\def\IntNode[#1][#2]{
 \filldraw[black] (#1) circle (#2*2pt);
}

\def\Anchor[#1][#2]{
 \filldraw[white] (#1) circle (#2*2pt);
 \draw[black] (#1) circle (#2*2pt);
}

\def\Leaf[#1][#2]{
 \filldraw[white] (#1)++(-#2*2pt,-#2*2pt) rectangle ++(#2*4pt,#2*4pt);
 \draw[black] (#1)++(-#2*2pt,-#2*2pt) rectangle ++(#2*4pt,#2*4pt);
}

\def\subtree{
  \tikz[scale=0.3]{
    \coordinate (a) at (0,0);
    \coordinate (b) at (-0.4,-1);
    \coordinate (c) at (0.4,-1);
    \draw (b) -- (a) -- (c);
    \foreach \p in {a} {\IntNode[\p][2]}
    \foreach \p in {b,c} {\Anchor[\p][2]}
  }
}

\title{Growing Binary Trees}

\author{Olivier Bodini
\institute{Sorbonne Paris Nord, CNRS, LAGA.}
\email{Olivier.Bodini@univ-paris13.fr}
\and
Antoine Genitrini \qquad Khaydar Nurligareev
\institute{Sorbonne Universit\'e, CNRS, LIP6, F-75005 Paris, France.}
\email{Antoine.Genitrini@lip6.fr \qquad Khaydar.Nurligareev@lip6.fr}
}

\def\titlerunning{Growing Binary Trees}
\def\authorrunning{O. Bodini, A. Genitrini \& K. Nurligareev}

\begin{document}
\maketitle

\begin{abstract}
 This paper introduces a new combinatorial framework for modeling the growth of binary trees through a discrete evolution process that incorporates a growing rule and an extinction rule. Building upon the theory of increasingly labeled structures and the analysis of polynomial iterates, we extend previous models of increasing trees with label repetitions by allowing growth branches to terminate. This mechanism enables a direct connection between dynamic evolutionary processes and classical unlabeled binary trees. We provide a combinatorial outlook for this model, linking our new approach to essential but traditionally complex parameters such as tree height, the maximum number of leaves at the deepest level (for a given tree size), and the overall tree profile. Our approach reveals structural links with Mandelbrot polynomials and coding theory. Furthermore, we leverage these structural insights to develop an efficient, iterative uniform random sampler for binary trees with a prescribed profile, achieving optimal complexity in both time and space and in random bit consumption.
\end{abstract}


\section{Introduction}

The combinatorial modeling of tree evolution serves as a fundamental framework for understanding the growth of discrete structures, bridging the gap between local recursive rules and global structural properties. Historically, a significant portion of the literature has been dedicated to labeled increasing trees, where nodes are integrated into a hierarchy according to specific temporal constraints. Early studies, such as the seminal work of Meir and Moon on recursive trees~\cite{MM78}, focused on the typical shape of structures generated by sequential addition. The analysis of functional equations for trees was formalized within the framework of species theory by Bergeron
\emph{et al.}~\cite{BLL88} and the references therein. Later the combinatorics of increasing trees has been extensively studied in the paper~\cite{BFS92}. Those structures comprise for example recursive trees, plane-oriented recursive trees, or binary increasing trees and have been studied through the lens of symbolic functional equations and differential equations. More recently, this research has been extended to include more complex labeling rules and structural synchronizations, notably through the study of increasing Schröder trees and strict monotonic trees~\cite{BGN19,BGMN22}. These models are particularly relevant in phylogenetics, where node labels encode the chronology of evolutionary branching, and exhibit deep connections to random graphs, such as the one-to-one correspondence between certain increasing Schröder trees and classical labeled graphs~\cite{BGN22}.

The mathematical genesis of these evolutionary models lies in the study of recursive processes where each new node is attached to an existing one according to a predefined rule. A seminal contribution to this area is the work of Flajolet and Odlyzko on iterates of polynomials~\cite{FO84} itself extending the paper~\cite{AS73}. Their analysis provided a robust framework for understanding the height and size of trees generated by iterative growth. This approach was then significantly extended in our previous work, On the number of increasing trees with label repetitions~\cite{BGGW20}, which explored discrete growth processes where multiple nodes can share the same label. These models, which involve synchronized branching events, typically lead to purely formal power series where the underlying increasing labeling is intrinsically tied to the structural growth.

In these previous frameworks, the growth process was essentially monotonic: the structures could only expand, and every leaf remained a potential site for further development. The present paper introduces a paradigm shift by adding an extinction (or death) rule to the evolution process. By allowing branches to terminate, we reconnect these dynamic growth models with classical unlabeled structures, specifically unlabeled binary trees. In this perspective, nodes are categorized into three types: internal nodes, active leaves (anchors), and dead leaves. This distinction allows us to consider essential but traditionally difficult-to-access parameters from a new perspective. By leveraging some symbolic substitution and the iterative dynamics of the resulting polynomials, we provide a combinatorial interpretation of the tree height, the maximum number of the deepest leaves (for a fixed tree size), and the tree profile (the number of leaves at each level). Finally, we leverage these structural insights to design a uniform random sampler for binary tree respecting a given profile, that is optimal in terms of time and space complexity, and random bits consumption.

The paper is organized as follows. In Section~\ref{sec:growth_model}, we formally define the growth process under consideration and describe the specific tree families and parameters that arise from this framework. Section~\ref{sec:a_n} focuses on the ``bushy'' trees (those with a maximum number of leaves at the maximum depth for a given tree size), analyzing their behavior as the overall tree size increases. In this context, we establish structural links to meta-Fibonacci sequences and coding theory.
In Section~\ref{sec:S_h}, we provide a comprehensive outlook on the growth dynamics of binary trees with respect to their height. Finally, leveraging our study of leaf distribution, Section~\ref{sec:sampling} explores the algorithmic implications of this model, specifically focusing on the uniform random sampling of binary trees satisfying a prescribed profile.

\section{Combinatorial modeling of a growth process}
\label{sec:growth_model}

\subsection{Growth process}

In our research, we focus on the specific family of binary trees whose nodes are of the following three types: \emph{internal nodes} ($\bullet$), active leaves or \emph{anchors} ($\circ$), and dead leaves or simply \emph{leaves} ({\tiny$\square$}).
Each tree is generated by the \emph{growth process} proceeding as follows.
\begin{enumerate}
  \item
	In the beginning (at the moment $t=0$), the tree is reduced to an anchor ($\circ$).
  \item
    At every moment $t\in\mathbb{Z}_{>0}$, we replace each anchor with a (dead) leaf ({\tiny$\square$}) or a subtree consisting of an internal node with two anchors attached as children 
    (\tikz[scale=0.3]{
      \coordinate (a) at (0,0);
      \coordinate (b) at (-0.4,-1);
      \coordinate (c) at (0.4,-1);
      \draw (b) -- (a) -- (c);
      \foreach \p in {a} {\IntNode[\p][2]}
      \foreach \p in {b,c} {\Anchor[\p][2]}
    }).
\end{enumerate}
We refer to a tree obtained after several steps of the growth process as a \emph{growing binary tree}.
Note, that the internal nodes always have two children in our model.
Such trees are sometimes called \emph{locally complete binary trees}.

From the construction, it is seen that the anchors appear exclusively on the last level.
An example of a~growing binary tree is shown in Figure~\ref{fig:tree-example}~a) at moment $t=5$.

 \begin{figure}[ht!]
  \centering
  \begin{tikzpicture}[scale = 0.6, line width=.5pt]
   \def\scl{1}
   \begin{scope}
    \coordinate (a1) at (0,0);
    \coordinate (b1) at (-1,-1);
    \coordinate (b2) at (1,-1);
    \coordinate (c1) at (-1.5,-2);
    \coordinate (c2) at (-0.5,-2);
    \coordinate (c3) at (0.5,-2);
    \coordinate (c4) at (1.5,-2);
    \coordinate (d1) at (-2.25,-3);
    \coordinate (d2) at (-0.75,-3);
    \coordinate (d3) at (0.75,-3);
    \coordinate (d4) at (2.25,-3);
    \coordinate (e1) at (-2.75,-4);
    \coordinate (e2) at (-1.75,-4);
    \coordinate (e3) at (-1.25,-4);
    \coordinate (e4) at (-0.25,-4);
    \coordinate (e5) at (0.25,-4);
    \coordinate (e6) at (1.25,-4);
    \coordinate (f1) at (-3,-5);
    \coordinate (f2) at (-2.5,-5);
    \coordinate (f3) at (-1.5,-5);
    \coordinate (f4) at (-1,-5);
    \coordinate (f5) at (0,-5);
    \coordinate (f6) at (0.5,-5);
    \coordinate (f7) at (1,-5);
    \coordinate (f8) at (1.5,-5);
    \draw (a1) -- (b1);
    \draw (a1) -- (b2);
    \draw (c1) -- (b1);
    \draw (c2) -- (b1);
    \draw (c3) -- (b2);
    \draw (c4) -- (b2);
    \draw (c1) -- (d1);
    \draw (c1) -- (d2);
    \draw (c4) -- (d3);
    \draw (c4) -- (d4);
    \draw (e1) -- (d1);
    \draw (e2) -- (d1);
    \draw (e3) -- (d2);
    \draw (e4) -- (d2);
    \draw (e5) -- (d3);
    \draw (e6) -- (d3);
    \draw (e1) -- (f1);
    \draw (e1) -- (f2);
    \draw (e3) -- (f3);
    \draw (e3) -- (f4);
    \draw (e5) -- (f5);
    \draw (e5) -- (f6);
    \draw (e6) -- (f7);
    \draw (e6) -- (f8);
    \foreach \p in {f1,f2,f3,f4,f5,f6,f7,f8} {\Anchor[\p][\scl]}
    \foreach \p in {a1,b1,b2,c1,c4,d1,d2,d3,e1,e3,e5,e6} {\IntNode[\p][\scl]}
    \foreach \p in {c2,c3,d4,e2,e4} {\Leaf[\p][\scl]}
    \draw (-3,-0.5) node {$a)$};
   \end{scope}
   \begin{scope}[xshift = 10cm]
    \coordinate (a1) at (-1,0);
    \coordinate (b1) at (-1,-1);
    \coordinate (c1) at (-1.5,-2);
    \coordinate (c2) at (-0.5,-2);
    \coordinate (d1) at (-2.25,-3);
    \coordinate (d2) at (-0.75,-3);
    \coordinate (e1) at (-2.75,-4);
    \coordinate (e2) at (-1.75,-4);
    \coordinate (e3) at (-1.25,-4);
    \coordinate (e4) at (-0.25,-4);
    \coordinate (f1) at (-3,-5);
    \coordinate (f2) at (-2.5,-5);
    \coordinate (f3) at (-1.5,-5);
    \coordinate (f4) at (-1,-5);
    \draw (a1) -- (b1);
    \draw (c1) -- (b1);
    \draw (c2) -- (b1);
    \draw (c1) -- (d1);
    \draw (c1) -- (d2);
    \draw (e1) -- (d1);
    \draw (e2) -- (d1);
    \draw (e3) -- (d2);
    \draw (e4) -- (d2);
    \draw (e1) -- (f1);
    \draw (e1) -- (f2);
    \draw (e3) -- (f3);
    \draw (e3) -- (f4);
    \foreach \p in {f1,f2,f3,f4} {\Anchor[\p][\scl]}
    \foreach \p in {a1,b1,c1,d1,d2,e1,e3} {\IntNode[\p][\scl]}
    \foreach \p in {c2,e2,e4} {\Leaf[\p][\scl]}
    \draw (-3,-0.5) node {$b)$};
   \end{scope}
  \end{tikzpicture}
  \caption{a) A growing binary tree of height $h=5$ with $m=8$ anchors. b) A growing binary tree with an additional internal node attached to its root.}
  \label{fig:tree-example}
 \end{figure}

\subsection{The set of growing binary trees}

Let us denote by $t_{n,m}$ the total number of growing binary trees with $n$ internal nodes and $m$ anchors.
From the description of the growth process, it is clear that the number of anchors of a tree is even with the exception of the initial tree.
Thus, $t_{n,2k-1}=0$ for all positive integers $n$ and $k$.
The first values of $t_{n,2k}$ are shown in Table~\ref{tab:t_n,2k}.

\begin{table}[ht!]
 \centering
 \scalebox{0.8}{
  $\begin{array}{c|cccccccccccccc}
   n & 1 & 2 & 3 & 4 & 5 & 6 & 7 & 8 & 9 & 10 & 11 & 12 & 13 & 14 \\
   \hline
  t_{n,2} & 1 & 2 & 4 & 12 & 32 & 104 & 328 & 1\,080 & 3\,648 & 12\,544 & 43\,600 & 153\,504 & 546\,272 & 1\,960\,368 \\ 
  t_{n,4} & 0 & 0 & 1 & 2 & 10 & 24 & 92 & 308 & 1\,028 &  3\,584 & 12\,736 & 45\,160 & 161\,152 & 581\,632 \\
  t_{n,6} & 0 & 0 & 0 & 0 & 0 & 4 & 8 & 40 & 176 & 584 & 2\,144 & 8\,192 & 30\,720 & 112\,496 \\
  t_{n,8} & 0 & 0 & 0 & 0 & 0 & 0 & 1 & 2 & 10 & 84 & 282 & 1\,048 & 4\,368 & 18\,224 \\
  t_{n,10} & 0 & 0 & 0 & 0 & 0 & 0 & 0 & 0 & 0 & 0 & 24 & 104 & 352 & 1\,616 \\
  t_{n,12} & 0 & 0 & 0 & 0 & 0 & 0 & 0 & 0 & 0 & 0 & 0 & 4 & 36 & 96 \\
  t_{n,14} & 0 & 0 & 0 & 0 & 0 & 0 & 0 & 0 & 0 & 0 & 0 & 0 & 0 & 8 \\
  \end{array}$
 }
 \caption{\label{tab:t_n,2k} The first numbers $t_{n,2k}$ of growing binary trees with $n$ internal nodes and $2k$ anchors.}
\end{table}

Note that there is no need to additionally count the number of (dead) leaves $\ell$, since it is completely determined by two other parameters:
if $m$ and $n$ are, respectively, the numbers of anchors and internal nodes in a growing binary tree, then $\ell = n - m + 1$.
Indeed, this relation holds for the initial tree, and the value $n-m-\ell$ is invariant under replacements of the growth process.

Let us further denote by
  \[
    T(x,z)
     :=
    x 
     +
    \sum\limits_{n=0}^{\infty}
     \sum\limits_{k=1}^{\infty} t_{n,2k}x^{2k}z^n
  \]
the generating function of \emph{active} growing binary trees, that is, trees with at least one anchor.
Here, the marking variables $x$ and $z$ count the anchors and internal nodes, respectively.
The series $T(x,z)$ satisfies the following relation:
\begin{equation}\label{eq:initial_gf_relation}
  T(x,z) = x + T(1+zx^2, z) - T(1,z)
  \, .  
\end{equation}
Indeed, replacing an anchor with a leaf corresponds symbolically to the substitution $x\mapsto1$, while a~replacement with a subtree\subtree is represented by the substitution $x\mapsto zx^2$.
Note that we have to deduce the term $T(1,z)$ in order to eliminate trees without anchors.
By substituting $x$ with the generating function $C(z)$ of Catalan numbers and taking into account that $C(z)=1+zC(z)^2$, we conclude that
\[
  T(1,z) = C(z)
  \, , \qquad\mbox{where}\quad
  C(z) = \dfrac{1-\sqrt{1-4z}}{2z}
  \, .  
\]
Combinatorially, this fact confirms that the growing binary trees without anchors are indeed all binary trees counted by the series $C(z)$.
In particular, the column sums in Table~\ref{tab:t_n,2k} are Catalan numbers:
\begin{equation*}
  \sum\limits_{m=1}^{n+1} t_{n,m} = C_n
  \, .
\end{equation*}

Another immediate consequence of Relation~\eqref{eq:initial_gf_relation} is the following recurrent formula for $t_{n,2k}$:
\begin{equation}\label{eq:initial_recurrence}
  t_{1,2} = 1
  \, , \qquad\quad
  t_{n,2k} = \sum\limits_{m=1}^{\infty} \binom{2m}{k} t_{n-k,2m}
  \quad\,\, \mbox{for } n>k\ge1
  \, .
\end{equation}
Relation~\eqref{eq:initial_recurrence} also has a combinatorial explanation:
to get a tree with $n$ internal nodes and $2k$ anchors, we have to replace $k$ anchors of the tree obtained in the previous step of the growth process with an internal node with two children; the other anchors become leaves.
The binomial coefficient $\binom{2m}{k}$ then represents the number of choices of $k$ anchors to be replaced.

\subsection{Cumulative anchor counting}

To enumerate the total number of leaves in active growing binary trees, we use the series $\tilde{T}(1,z)$, where
\[
  \tilde{T}(x,z) := 
  \dfrac{\partial}{\partial x}T(x,z)
   =
  1 
   +
  \sum\limits_{n=0}^{\infty}
   \sum\limits_{k=1}^{\infty} 2kt_{n,2k}x^{2k-1}z^n
  \, .
\]
By differentiating Relation~\eqref{eq:initial_gf_relation} with respect to $x$, we obtain the following relation for $\tilde{T}(x,z)$:
\begin{equation}\label{eq:tilde(T)_relation}
  \tilde{T}(x,z) = 1 + 2xz\tilde{T}(1+zx^2, z)
  \, .  
\end{equation}
This relation allows us to express the series $\tilde{T}(x,z)$ in terms of polynomials $p_h(x,z)$ that count all growing binary trees that can be obtained after the first $h$ steps of the growing process.
In fact, the polynomials $p_h(x,z)$ are determined by the following recurrence relation:
\begin{equation}\label{eq:p_h(x,z)}
  p_0(x,z) = x
  \, , \qquad\quad
  p_{h+1}(x,z) = p_h(1+zx^2,z)\quad\,\, \mbox{for } h\ge0
  \, .  
\end{equation}
Therefore, taking into account that $1+zx^2=p_1(x,z)$, we can iterate Relation~\eqref{eq:tilde(T)_relation} according to~\eqref{eq:p_h(x,z)}.
By passing to the limit, this gives us the following expression:
\[
  \tilde{T}(x,z) = 1 + 
  \sum\limits_{i=1}^{\infty} (2z)^i
   \prod\limits_{j=0}^{i-1}p_j(x,z)
  \, .  
\]
Note that the polynomials $p_h(x,z)$ are related to the so-called (shifted) \emph{Mandelbrot polynomials}~\cite{CCC21}:
\begin{equation*}
  M_h(z) = z\, q_h(1,z)
  \, .
\end{equation*}
From a combinatorial point of view, Mandelbrot polynomials count growing binary trees with an additional internal node attached to the root; see Figure~\ref{fig:tree-example}~b).
By substituting $x=1$ into the relation for $\tilde{T}(x,z)$, we obtain an expression for $\tilde{T}(1,z)$ in terms of Mandelbrot polynomials:
\[
  \tilde{T}(1,z) = 1 + 
  \sum\limits_{i=1}^{\infty} 2^i
   \prod\limits_{j=0}^{i-1}M_j(z)
  \, .  
\]
\begin{remark}
The formal link to Mandelbrot polynomials highlights a phase transition in the tree growth. The discrete dynamic of the iterates $p_h(1,z)$ has fixed points satisfying $x = 1 + zx^2$, which results in a~critical bifurcation at $z_c = 1/4$. For $|z| < 1/4$, the process contracts towards the Catalan generating function $C(z)$ (almost-sure extinction), whereas for $z > 1/4$, it diverges (explosive growth). Consequently, the analyticity domain of the process is intrinsically bounded by the main cardioid of the Mandelbrot set.
\end{remark}

\section{Maximal number of anchors for a given tree size}
\label{sec:a_n}

Let us turn our focus to the columns of Table~\ref{tab:t_n,2k}.
We can see that in every column the number of nonzero elements is finite.
For the $n$th column, we denote this number $a_n$.
In other words,
\begin{equation*}
  a_n = \max\{k \colon t_{n,2k}>0\}
  \, .
\end{equation*}
The first several values of the sequence $(a_n)_{n\ge1}$ are
\[
 (a_n)_{n\ge 1} = 1, 1, 2, 2, 2, 3, 4, 4, 4, 4, 5, 6, 6, 7, 8, 8, 8, 8, 8, 9, 10, 10, 11, 12, 12, 12, 13, \ldots
\]

The purpose of this section is to provide a description of the sequence $(a_n)_{n\ge1}$ and to show that this sequence is actually \href{https://oeis.org/A006949}{A006949} from~\cite{oeis}.
First, we show that $(a_n)_{n\ge1}$ can be defined independently of growing binary trees.

\begin{lemma}\label{lemma:a_n-recurrence}
  The sequence $(a_n)_{n\ge1}$ satisfies the following recurrent relation:
  \begin{equation*}
    a_1 = 1
    \, , \qquad\quad
    a_n = \max\{k\colon k \le 2a_{n-k}\}
    \quad\,\, \mbox{for } n>1
    \, .
  \end{equation*}
\end{lemma}
\begin{proof}
  Using the sequence $(a_n)_{n\ge1}$, Relation~\eqref{eq:initial_recurrence} can be rewritten as
  \[
    t_{n,2k} = \sum\limits_{m=\lceil k/2\rceil}^{a_{n-k}} \binom{2m}{k} t_{n-k,2m}
    \, .
  \]
  Hence, the condition $t_{n,2k} \neq 0$ means that there exists a positive integer $m$ such that $k \le 2m \le 2a_{n-k}$,
  which is equivalent to the condition $k \le 2a_{n-k}$.
\end{proof}

From Lemma~\ref{lemma:a_n-recurrence}, it follows by induction that the sequence $(a_n)_{n\ge1}$ is increasing in the sense that, for any $n\in\mathbb{Z}_{>0}$, we have $a_{n+1} = a_n$ or $a_{n+1} = a_n+1$.
This fact, coupled with the initial condition $a_1=1$, means that the behavior of the sequence $(a_n)_{n\ge1}$ is completely determined by the number of repetitions:
\begin{equation*}
  a_n = \max\left\{k \colon\, \sum\limits_{i=1}^{k} b_k \ge n\right\}
  \, ,
\end{equation*}
where $b_n$ is the number of elements in $(a_n)_{n\ge1}$ that are equal to $n$:
\begin{equation*}
  b_n = \#\{k \colon a_k=n\}
  \, .
\end{equation*}
The first several values of the sequence $(b_n)_{n\ge1}$ are
\[
 (b_n)_{n\ge 1} = 2, 3, 1, 4, 1, 2, 1, 5, 1, 2, 1, 3, 1, 2, 1, 6, 1, 2, 1, 3, 1, 2, 1, 4, 1, 2, 1, \ldots
\]

\begin{proposition}\label{prop:b_n-formula}
  The sequence $(b_n)_{n\ge1}$ satisfies the following relation:
  \begin{equation*}
    b_n
     = 
    \left\{\begin{array}{ll}
      p+2\quad & \mbox{if } n=2^p, \\
      p+1\quad & \mbox{if } n=2^pa,\,\, a\mbox{ is odd},\,\, a>1\, . \\
    \end{array}\right.
  \end{equation*}
  In particular, we have
  $b_{2n}=b_n+1$ for even indices,
  $b_{2n+1}=1$ for odd indices greater than $1$,
  and $b_1=2$.
\end{proposition}


Proposition~\ref{prop:b_n-formula} allows us to claim that $(a_n)_{n\ge1}$, with an additional value $a_0=1$, coincides with the entry \href{https://oeis.org/A006949}{\texttt{A006949}} from \textsc{oeis}~\cite{oeis} and is known as a \emph{meta-Fibonacci sequence} (for the sequence $(b_n)_{n\ge1}$, see the entries \href{https://oeis.org/A135560}{\texttt{A135560}} and \href{https://oeis.org/A241235}{\texttt{A241235}}).
In particular, due to~\cite{T92} and~\cite{RD09}, we obtain the following properties of this sequence.

\begin{corollary}\label{cor:a_n_properties}
  The sequence $(a_n)_{n\ge1}$ satisfies the following recurrent relation:
  \begin{equation*}
    a_0 = a_1 = a_2 = 1
    \, , \qquad\quad
    a_n = a_{n-1-a_{n-1}} + a_{n-2-a_{n-2}}
    \quad\,\, \mbox{for } n>2
    \, .
  \end{equation*}
  Its generating function and asymptotic behavior satisfy, respectively,
  \begin{equation*}
    \sum\limits_{n=0}^{\infty} a_n z^n
     = 
    z\sum\limits_{n=0}^{\infty}
     \prod\limits_{i=1}^{n} (z + z^{2^i})
    \qquad\mbox{and}\qquad
    \lim\limits_{n\to\infty}\dfrac{a_n}{n} = \dfrac{1}{2}
    \, .
  \end{equation*}
\end{corollary}

Note that the sequence $(a_n)_{n\ge1}$ also admits another combinatorial interpretation in terms of infinite binary trees~\cite{RD09,D11} and is related to extremal compact codes~\cite{JR06}.

\section{Growing binary trees of fixed height}
\label{sec:S_h}

In this section, we study the behavior of growing binary trees with respect to their heights.

\begin{lemma}\label{lem:parameters_relations}
  For any active growing binary tree, the number $n$ of its internal nodes, the number $2k$ of its anchors, and its height $h$ satisfy
  \begin{equation*}
    h \le n - k + 1 \le 2^{h-1}
    \, .
  \end{equation*}
  Both inequalities are sharp.
\end{lemma}

Let us denote by $t_{n,m,h}$ the number of active growing binary trees of height $h$ with $n$ internal nodes and $m$ anchors.
Clearly, apart from the case $h=0$, the number of anchors in a tree of height $h$ is even.
For example, for $h=4$, the values of $t_{n,2k,h}$ are shown in Figure~\ref{fig:h4}.

\begin{figure}[ht!]
  \centering
\begin{tikzpicture}[scale = 0.5, line width=.5pt, >=latex]
 \def\GroundColor{yellow!50!white}
 \def\FrontColor{red!50!white}
 \def\BackColor{blue!50!white}
 \def\FrontBackColor{purple!50!white}
 \def\MiddleColor{green!50!white}
 \def\Block[#1][#2][#3][#4]{
  \begin{scope}[xshift=#1*28.5, yshift=#2*28.5]
   \filldraw[#4] (0,0) rectangle ++(-1,-1);
   \draw (-0.5,-0.5) node {$#3$};
  \end{scope}
 }
 \begin{scope}
  \def\n{15}
  \def\k{8}
  \filldraw[\GroundColor] (0,0) rectangle (\n,\k);
  \draw[->] (\n,0) -- ++(1,0);
  \draw[->] (0,\k) -- ++(0,1);
  \draw (\n+0.5,-0.5) node {$n$};
  \draw (-0.5,\k+0.5) node {$k$};
  \foreach \p in {1,...,\n} \draw (\p-0.5,-0.5) node {$\p$};
  \foreach \q in {0,...,\k} \draw (-0.5,\q-0.5) node {$\q$};
  \foreach \p in {1,...,\n}{
   \foreach \q in {1,...,\k}{
    \draw (\p-0.5,\q-0.5) node {$0$};
   }
  }
  \Block[4][1][8][\FrontColor]
  \Block[5][2][4][\FrontColor]
  \Block[5][1][16][\MiddleColor]
  \Block[6][2][16][\FrontColor]
  \Block[7][3][8][\FrontColor]
  \Block[8][4][2][\FrontColor]
  \Block[6][1][24][\MiddleColor]
  \Block[7][2][36][\MiddleColor]
  \Block[8][3][24][\MiddleColor]
  \Block[9][4][6][\FrontColor]
  \Block[7][1][24][\MiddleColor]
  \Block[8][2][60][\MiddleColor]
  \Block[9][3][80][\MiddleColor]
  \Block[10][4][60][\FrontColor]
  \Block[11][5][24][\FrontColor]
  \Block[12][6][4][\FrontColor]
  \Block[8][1][8][\BackColor]
  \Block[9][2][28][\BackColor]
  \Block[10][3][56][\BackColor]
  \Block[11][4][70][\BackColor]
  \Block[12][5][56][\BackColor]
  \Block[13][6][28][\FrontBackColor]
  \Block[14][7][8][\FrontBackColor]
  \Block[15][8][1][\FrontBackColor]
  \draw (0,0) grid (\n,\k);
 \end{scope}
\end{tikzpicture}
  \caption{Numbers $t_{n,2k,h}$ of growing binary trees of height $h=4$.}
  \label{fig:h4}
\end{figure}

Given a positive integer $h$, we also introduce the domain $S_h$ of nonzero values of $t_{n,2k,h}$:
  \begin{equation}\label{eq:nonzero_domain_S_h}
    S_h = \{(n,k)\colon t_{n,2k,h} \neq 0\}
    \, .
  \end{equation}
From Lemma~\ref{lem:parameters_relations}, it follows that $S_h$ is finite.
The following results describe its behavior. 

\begin{proposition}\label{prop:Gamma_h_exact_form}
 The right boundary $\Gamma_h$ of the nonzero domain $S_h$ (blue cells in Figure~\ref{fig:h4}) satisfies
 \begin{equation}\label{eq:Gamma_h_exact_form}
  \Gamma_h = \{(2^{h-1}-1+i,i)\colon 1\le i \le 2^{h-1}\}
  \, .
 \end{equation}
 In particular, all elements of $\Gamma_h$ belong to the line $n-k = 2^{h-1}-1$.
\end{proposition}

\begin{proposition}\label{prop:Lambda_h_exact_form}
 The upper boundary $\Lambda_h$ of the nonzero domain $S_h$ (red cells in Figure~\ref{fig:h4}) satisfies
 \begin{equation}\label{eq:Lambda_h_exact_form}
  \La_h = \{
   (h,\hat{a}_2),
   (h+1,\hat{a}_3),
   \ldots,
   (2^{h}-2,\hat{a}_{2^h-h}),
   (2^{h}-1,\hat{a}_{2^h-h+1})
  \}
  \, ,
 \end{equation}
 where
 \begin{equation*}
  \hat{a}_n = \max\left\{k \colon\, \sum\limits_{i=1}^{k} \hat{b}_k \ge n\right\},
 \end{equation*}
 and $\hat{b}_n$ is the \emph{ruler function}, that is, the number of times $2n$ can be divided by two.
\end{proposition}

\begin{corollary}\label{cor:|S_h|}
 The area of the nonzero domain $S_h$ is equal to $2^{h-2}(2^{h-1}-h+2)$.
\end{corollary}

\begin{proposition}\label{prop:S_h-scaling-limit}
 The scaling limit of the nonzero domain $S_h$ (normalized by $2^{h-1}$), as $h\to\infty$, is the triangle with vertices $(0,0)$, $(1,0)$, and $(2,1)$. 
\end{proposition}

\section{Random sampling}
\label{sec:sampling}

\subsection{Combinatorial context}

We are focusing on binary trees, where the atoms under consideration are the leaves. 
An analogous approach can be applied to binary trees according to their number of internal nodes.

\begin{definition}
    The \emph{level} of a node in a binary tree is the distance (in terms of traversed edges) from the root to that node. The root is at level $0$.
\end{definition}

\begin{definition}
    The profile of a binary tree is the sequence $(\ell_0, \ell_1, \ldots, \ell_h)$ of the
    nonnegative numbers of leaves at each level.
    The index $h$ corresponds to the \emph{height} of the tree. The \emph{size} of the tree 
    (\emph{i.e.} the total number of leaves it contains) is equal to the sum $\sum_i \ell_i$.
\end{definition}
For a tree reduced to a single leaf (the root), we get the profile $(1)$.
For any other binary tree with height $h > 0$, we have $\ell_0 = 0$ (the root is not a leaf).

\begin{fact}[Kraft-McMillan equality]
    There exists a binary tree with profile $(\ell_0, \ell_1, \ldots, \ell_h)$
    if and only if
    \[
        \sum_{i=0}^h \frac{\ell_i}{2^{i}} = 1.
    \]
\end{fact}
\noindent This is a classical result in the context of code theory; see, for instance, \cite{CT06}.

\begin{proposition}
\label{prop:intern-profile}
    Let $h$ be a positive integer and $T$ be a binary tree with profile $(\ell_0, \ell_1, \ldots, \ell_h)$.
    Its internal node profile (\emph{i.e.} the sequence of the number of internal nodes at each level) is given by $(i_0 , i_1, \ldots, i_{h-1})$ satisfying 
    \[
        \begin{cases}
            i_0 = 1 \\
            i_k = 2i_{k-1} - \ell_k, \quad \text{for } 1 \le k \le h-1.
        \end{cases}
    \]
\end{proposition}

\begin{definition}
    We say that a profile $(\ell_0, \ell_1, \ldots, \ell_h)$ is \emph{valid} if
    there exists at least one binary tree having this profile.
\end{definition}

\begin{proposition}
\label{prop:counting_trees}
    There exists a unique binary tree, which is reduced to a single leaf, having the profile~$(1)$.
    Let $(\ell_0, \ell_1, \ldots, \ell_h)$, with $h\ge 1$, be a valid profile.
    The number of binary trees having this profile is 
    \[
        \prod_{k=0}^{h-1} \binom{2i_{k}}{\ell_{k+1}},
    \]
    where the sequence $(i_k)_{k=0\dots h-1}$ is defined in Proposition~\ref{prop:intern-profile}.
\end{proposition}

\subsection{Uniform random sampling of trees with a given profile}

A natural way to sample binary trees with a given profile is the following. Start with the root of the tree;
if the latter is a leaf, then stop; otherwise, we have to decide 
what are the valid profiles of the two children. In the classical recursive method~\cite{NW75,FZVC94},
even for sampling trees of a given size, the necessary computations are heavy.
In our case, the combinatorial complexity is much higher, since we have to consider the profiles of the two subtrees,
which are not independent (like in the case of size) but are also constrained to be valid profiles.
In the context of binary decision diagram sampling~\cite{CG20}, such a~top-down approach has been addressed
to generate structures with a given profile, but it has finally been overtaken by an iterative
approach~\cite{CG23}, level by level, only focusing on a given level at a time without the need to consider the rest of the profile.

Here in the context of binary trees, such an iterative approach can be developed as well
(and will be more efficient than the one for decision diagrams), but an even more efficient approach is possible, 
looking to the tree from the bottom to the top (and for a given level, from left to right).

We can sample a binary tree with a given profile by starting from the last level, 
and then iteratively going up to the root, level by level, and for each level.
In fact, in Proposition~\ref{prop:counting_trees}, we can look to the product formula counting the trees with a given profile
from the last factor to the first ones, \emph{i.e.} from the deepest levels of the trees to the root level.
\begin{algorithm}[H]
 \begin{algorithmic}[1]
    \renewcommand{\algorithmicrequire}{\textbf{Input:}}
    \renewcommand{\algorithmicensure}{\textbf{Output:}}
    \Require $L = (\ell_0, \ell_1, \ldots, \ell_h)$ 
    \Ensure a uniform binary tree with profile $L$
    \Function{uniform\_tree}{$L$}
        \State $T$ is a sequence of $\ell_h/2$ nodes, each one pointing to $2$ leaves
        \For{$i$ from $h-1$ to $0$ with step $-1$}
          \State $T'$ = shuffling $T$ with a sequence of $\ell_i$ leaves
          \State $T$ is a sequence of $len(T')/2$ nodes with the $i$-th node pointing to nodes $2i$ and $2i+1$ of $T'$ 
      \EndFor
      \State \textbf{return} first (and single) node of $T$, corresponding to
      the root of the tree
    \EndFunction
	\end{algorithmic}  
	{\footnotesize
			\begin{flushleft}
				Function $len$ returns the number of elements of a sequence.
			\end{flushleft}	
		}
 \caption{Uniform random sampling of a tree with profile $(\ell_0, \ell_1, \ldots, \ell_h)$}
 \label{algo:unif_tree}
\end{algorithm}
The shuffling operation is presented in detail in Algorithm 3 from~\cite{BDGP17}.
\begin{proposition}
  Algorithm~\ref{algo:unif_tree} is correct and optimal in time, space, and random bit consumption.
\end{proposition}
Optimality stands for the time complexity, 
for the space complexity, and for the consumption of random bits (whose derivation comes from Proposition~\ref{prop:counting_trees} and the shuffling operation from~\cite{BDGP17}).


\section{Perspectives}

As a primary perspective, the connection between Mandelbrot polynomials and the framework of Flajolet and Odlyzko~\cite{FO82} allows for a precise characterization of the limit distribution of the deepest leaves. This approach provides a new way to analyze the tree's boundary and could be generalized to other substitution rules, offering a unified combinatorial view of extinction-based growth processes.

Algorithmically, the iterative bottom-up sampling introduced in Section~\ref{sec:sampling} can be broadly extended. Natural next steps include the uniform generation of binary forests and the sampling of trees with partial profiles (e.g., profiles with "holes" representing intermediate levels of arbitrary sizes). 

\newpage

\paragraph{Acknowledgments}
This work is partially funded by \textsc{anr-fwf} project \textsc{PAnDAG} ANR-23-CE48-0014 and \textsc{anr} project \textsc{COMETA-GAE} ANR-25-CE48-0602. 

\bibliographystyle{eptcs}
\bibliography{Growing_binary_trees}



\end{document}